\documentclass{amsart} 
\usepackage[a4paper,centering,textwidth=145mm,textheight=220mm]{geometry}
\usepackage{amssymb}
\usepackage{amsmath}
\usepackage{color}

\newtheorem{prop}{Proposition}[section] 

\newtheorem{thm}[prop]{Theorem} 
\newtheorem{cor}[prop]{Corollary}

\theoremstyle{remark} 
\newtheorem{rem}[prop]{Remark}

\newtheorem{ex}[prop]{Example}

\theoremstyle{definition}

\newcommand{\mz}{{\mathbb Z}}

\newcommand{\qform}[1]{{\langle{#1}\rangle}} % formes quadratiques
\newcommand{\pform}[1]{{\langle\!\langle{#1}\rangle\!\rangle}} %formes de Pfister

\DeclareMathOperator{\End}{End}

\DeclareMathOperator{\id}{Id}
\DeclareMathOperator{\Int}{Int}
\DeclareMathOperator{\Nrd}{Nrd}

\newcommand{\Spin}{\mathrm {Spin}}

\DeclareMathOperator{\ad}{ad} % adjoint involution
\DeclareMathOperator{\Ad}{Ad}
\newcommand{\ba}{\overline{\rule{2.5mm}{0mm}\rule{0mm}{4pt}}} %canonical involution
\title[Decomposability of orthogonal involutions in degree
$12$]{Decomposability of orthogonal involutions in degree $12$}  
\author[A.~Qu\'eguiner-Mathieu]{Anne Qu\'eguiner-Mathieu}
\address{LAGA, UMR 7539\\ CNRS\\ Universit\'e Paris 13 - Sorbonne Paris Cit\'e\\ Universit\'e Paris 8\\ F-93430 Villetaneuse, France}
\email{queguin@math.univ-paris13.fr}

\author[J.-P.~Tignol]{Jean-Pierre Tignol}
\address{UCLouvain\\ Institut ICTEAM, Box L4.05.01\\ B-1348
  Louvain-la-Neuve, Belgique}
\email{jean-pierre.tignol@uclouvain.be}
\thanks{The second author acknowledges support from the Fonds de la
  Recherche Scientifique--FNRS under grant n$^\circ$~J.0159.19.}
\date{10 July 2019}


\begin{document}

\maketitle
\begin{abstract} 
A theorem of Pfister asserts that every
$12$-dimensional quadratic form with trivial discriminant and trivial
Clifford invariant over a field of characteristic different from~$2$
decomposes as a tensor product of a binary quadratic form and a
$6$-dimensional quadratic form with trivial discriminant.
The main result of the paper extends Pfister's result to orthogonal involutions : every central simple algebra of degree~$12$ with orthogonal involution
of trivial discriminant and trivial Clifford invariant decomposes into
a tensor product of a quaternion algebra and a central simple algebra
of degree~$6$ with orthogonal involutions. This decomposition is used
to establish a criterion for the existence of orthogonal involutions
with trivial invariants on algebras of degree~$12$, and to calculate
the $f_3$-invariant of the involution if the algebra has
index~$2$.
\end{abstract}

Every semi-simple algebraic group of classical type can be described in terms of a central simple algebra with involution, except for groups of type $D$ in characteristic $2$, where the involution should be replaced by a so-called quadratic pair~\cite[\S 26]{KMRT}. When the base field has characteristic $0$, this was first observed by Weil~\cite{W} in the 60's for adjoint groups. In particular, over a field of characteristic different from $2$, groups of type $D_n$ are quotients of the $\Spin$ group of a degree $2n$ algebra with orthogonal involution. If the algebra is the endormophism ring of some $2n$-dimensional vector-space $V$, the involution is adjoint to a quadratic form $q$ defined on $V$, unique up to a scalar factor, and the corresponding groups are quotients of the $\Spin$ group of this quadratic form. 

Algebraic groups of low rank, and the corresponding algebras with involution, which have degree $\leq 14$, play a special role in the theory. Indeed, these groups have specific properties, which in turn produce efficient tools to study and describe the underlying algebraic objects. In particular, we may mention the so-called exceptional isomorphisms, with consequences on algebras with involution explored in~\cite[\S\,15]{KMRT}, triality, that is the action of the symmetric group in three letters on the Dynkin diagram $D_4$, see~\cite[Chapter X]{KMRT}, and the existence of an open orbit for some representations of algebraic groups of low rank, allowing to view torsors under those groups as torsors under the stabilizer, see Garibaldi~\cite[Th.~9.3]{Ga:Lens}. 

Even though they were first studied independently, these facts are related to the classification theorems describing quadratic forms of even dimension $\leq 12$ with trivial discriminant and trivial Clifford invariant, which were proved by Pfister in 1966~\cite{Pfister}, see also~\cite[Th.~8.1.1]{Kahn}. It appears that those forms always contain a nontrivial subform of even dimension and trivial discriminant, and admit a diagonalisation of a special shape, depending on the dimension of the form. In particular, the number of parameters required to describe such a form in general is less than what one may expect in view of the dimension. An analogous statement was obtained by Rost~\cite{Rost}, more than thirty years later, for quadratic forms of dimension $14$ (see also~\cite[Th.~21.3]{Ga:Lens}), based on the representation argument mentioned above. From the point of view of algebraic groups, it is clear that those results do not extend to higher dimensional quadratic forms. This was formally proved by Merkurjev and Chernousov in~\cite{CM}, where they compute the essential dimension of a split spinor group $\Spin_n$, for $n\geq 15$. 
Roughly speaking, since torsors under $\Spin_n$ are closely related to $n$-dimensional quadratic forms with trivial discriminant and trivial Clifford invariant, this essential dimension provides a measure of the number of parameters required to describe such a form in general. It follows from this computation that a general quadratic form of dimension $\geq 15$ does not contain a subform of a given dimension and with trivial discriminant, with two possible exceptions (see~\cite[Th.~4.2]{CM} for a precise statement).

As opposed to this, Pfister's theorem does extend to algebras with orthogonal involutions. This was already known in dimension $\leq 10$, and partial results in dimension $12$ were discussed in~\cite{GaQ} and~\cite{QT:Peyre}. The main result of this paper is Theorem~\ref{Pfister.thm}, which is an improved version of these
dimension $12$ analogues, obtained by using the descent theorem
for unitary involutions in degree~$6$ proven
in~\cite[Th.~1.3]{QT:ext}. 
This new statement is closer to Pfister's original result, which asserts that every
$12$-dimensional quadratic form with trivial discriminant and trivial
Clifford invariant over a field of characteristic different from~$2$
decomposes as a tensor product of a binary quadratic form and a
$6$-dimensional quadratic form with trivial discriminant. 

As a consequence, we characterize in
Corollary~\ref{cor:existence} the biquaternion $F$-algebras $D$ such
that the matrix algebra $M_3(D)$ carries an orthogonal involution with
trivial discriminant and trivial Clifford algebra. This property turns
out to hold for every biquaternion $F$-algebra if the
$2$-cohomological dimension of $F$ is at most~$2$; we show in
Example~\ref{ex:existence} that it fails for certain totally ramified
biquaternion $F$-algebras.

Another use of Theorem~\ref{Pfister.thm} is for the computation of a
certain cohomological invariant. Recall from~\cite{QT:Peyre} that a
cohomological invariant of degree~$3$ for orthogonal involutions 
with trivial discriminant and trivial Clifford invariant is defined
on the model of the Arason invariant $e_3$ of
quadratic forms. The generalized Arason invariant takes its values in
a quotient of the third Galois cohomology group of $\mu_4^{\otimes2}$;
taking the square of a representative yields an invariant $f_3$ with
values in the cohomology of $\mu_2$. We show in Theorem~\ref{f3.thm}
how this invariant can be calculated from a tensor product
decomposition afforded by Theorem~\ref{Pfister.thm}.
\medbreak

Throughout, $F$ is a field of characteristic different from~$2$, and
$(A,\sigma)$ is a central simple $F$-algebra with orthogonal
involution. A possible characterization of $(A,\sigma)$ is the existence of a 
finite Galois extension $L/F$ and a quadratic space $(V,\varphi)$ over $L$ such that 
\[A\otimes_F L\simeq \End_L(V)\mbox{ and }\sigma\otimes\id=\ad_\varphi\] where $\ad_\varphi$ is the involution 
adjoint to $\varphi$ (or, more
precisely, to its polar bilinear form). 
We generally follow the notation used in \cite{KMRT}, to
which we refer for background information on involutions on central
simple algebras. In particular, for any field $K$ containing $F$, we
write $(A,\sigma)_K$ for the $K$-algebra with involution $(A\otimes_F K,\sigma\otimes\id)$. If $\varphi$ is a (nondegenerate)
quadratic form on some $F$-vector space $V$, we write $\Ad_\varphi$ for
$(\End_F(V),\ad_\varphi)$. The discriminant of a quadratic form and the  even part of its Clifford algebra, which are invariant under similitudes, and may therefore be considered as invariants of the involution $\ad_\varphi$, extend to non-split algebras with orthogonal involution~\cite[\S 7,8]{KMRT}.

For $i\geq1$, we let $H^i(F)$ denote the Galois cohomology group
$H^i(F,\mu_2)$ and identify $H^1(F)$ with $F^\times/F^{\times2}$
(written additively) and $H^2(F)$ with the $2$-torsion subgroup of the
Brauer group of $F$. For $a\in F^\times$ and $A$ a central simple
$F$-algebra of exponent~$1$ or $2$, we write $(a)$ for the
square-class of $a$ and $[A]$ for the Brauer class of $A$. For every
orthogonal involution $\sigma$ on a central simple $F$-algebra $A$ of
even degree, we
let $e_1(\sigma)\in H^1(F)$ denote the discriminant of $\sigma$. If
$e_1(\sigma)=0$, the Clifford invariant $e_2(\sigma)\in
H^2(F)/\{0,[A]\}$ is the coset represented by any of the two
components of the Clifford algebra $C(A,\sigma)$. 

\section{Decomposability}

Our first decomposition result does not require triviality of the
Clifford invariant. It is premised instead on the existence of a
quadratic extension making the involution hyperbolic, i.e., adjoint to
a hyperbolic hermitian form.

\begin{prop} 
\label{dec.prop}
Let $(A,\sigma)$ be a central simple $F$-algebra with orthogonal
involution of degree~$12$, and let $K=F(\sqrt d)$ be a quadratic field
extension of $F$. If $A$ is split, assume additionally that $\sigma$
is not adjoint to a quadratic form of odd Witt index.
\begin{enumerate}
\item[(i)]
  The algebra with involution $(A,\sigma)_K$ is hyperbolic if and only
  if $(A,\sigma)$ decomposes as
  \[
    (A,\sigma) = (A_0,\sigma_0)\otimes(H,\rho)
  \]
  where $(A_0,\sigma_0)$ is a central simple algebra with orthogonal
  involution of degree~$6$ and $(H,\rho)$ is a quaternion algebra with
  orthogonal involution of discriminant~$(d)$.
\item[(ii)]
  The algebra with involution $(A,\sigma)_K$ is split and hyperbolic
  if and only if $(A,\sigma)$ decomposes as
  \[
    (A,\sigma)={\Ad_\varphi}\otimes (H,\rho)
  \]
  where $\varphi$ is a quadratic form of dimension~$6$ and $(H,\rho)$
  is a quaternion algebra with orthogonal involution of
  discriminant~$(d)$. 
\end{enumerate}
\end{prop}

\begin{proof}
  (i)
The condition is obviously sufficient, since $(H,\rho)_K$ is
hyperbolic.  
Assume conversely that $(A,\sigma)_K$ is
hyperbolic. By~\cite[Th.~3.3]{BST}, this means $A$ contains a
skew-symmetric element $\delta$ 
with square $d$. Writing $\iota$ for the nontrivial automorphism of
$K$, we may then identify $(K,\iota)$ with a subalgebra of 
$(A,\sigma)$. Let $B$ be the centralizer of $K$ in $A$. The involution
$\sigma$ induces an involution $\tau$ of $B$, which restricts to
$\iota$ on $K$.  Hence by the descent theorem
of~\cite[Th.~1.3]{QT:ext}, $(B,\tau)=(A_0,\sigma_0)\otimes_F
(K,\iota)$, for some algebra with orthogonal involution
$(A_0,\sigma_0)$. The centralizer of $A_0$ in $A$ is a quaternion
algebra $H$, which contains $K$, and by the double centralizer
theorem, we have $A=A_0\otimes H$. Moreover, since $A_0$ is
$\sigma$-stable, $H$ also is, and we get a decomposition  
\[
  (A,\sigma)=(A_0,\sigma_0)\otimes (H,\rho),
\]
with $\sigma_0$ and $\rho$ of orthogonal type, and $(H,\rho)\supset
(K,\iota)$. The latter inclusion shows that $e_1(\rho)=(d)$, and the
proof of~(i) is complete.

(ii)
As in~(i), the condition is sufficient because $(H,\rho)_K$ is
hyperbolic. For the converse, we modify the argument in~(i), taking
into account the additional hypothesis that $A_K$ is split. From this
hypothesis, it follows that the algebra $B$ is split, hence we may
identify $B=\End_K(V)$ for some $K$-vector space $V$, and $\tau=\ad_h$
for some hermitian form $h$ on $V$. Fix an orthogonal basis
$(e_1,\ldots,e_6)$ of $V$. The form $h$ restricts to a symmetric
bilinear form on the $F$-vector space $V_0$ spanned by $e_1$, \ldots,
$e_6$, and we may take $A_0=\End_F(V_0)$ in the proof of~(i). Thus,
$(A_0,\sigma_0)=\Ad_\varphi$ where $\varphi(x)=h(x,x)$ on $V_0$. 
\end{proof}

\begin{rem}
Let $(A,\sigma)$ be a central simple $F$-algebra with orthogonal 
involution of degree~$4m$ for some integer $m$ (excluding the case where $A$ is split and $\sigma$ is adjoint to a quadratic form of odd Witt index). We compare the following statements:
\begin{enumerate}
\item[(a)]
$(A,\sigma)=(A_0,\sigma_0)\otimes(H,\rho)$ for some quaternion algebra with orthogonal involution $(H,\rho)$;
\item[(b)]
there exists a quadratic field extension $K$ of $F$ such that $(A,\sigma)_K$ is hyperbolic;
\item[(c)]
$e_1(\sigma)=0$.
\end{enumerate}
The implication (a)~$\Rightarrow$~(b) always holds, for we may take for $K$ the subfield of $H$ generated by a skew-symmetric element. (If the skew-symmetric elements in $H$ do not generate a field, then $(H,\rho)$ is hyperbolic and (b) clearly holds.) 
The implication (b)~$\Rightarrow$~(c) can be derived from the first step in the proof of Proposition~\ref{dec.prop} as follows: if $(A,\sigma)_K$ is hyperbolic, then $(K,\iota)$ embeds in $(A,\sigma)$ by \cite[Th.~3.3]{BST}, hence $A$ contains a skew-symmetric element $\alpha$ such that $\alpha^2\in F^\times$. Let $\alpha^2=a$. The reduced norm $\Nrd_A(\alpha)$ is $(-a)^{2m}$ and by definition $e_1(\sigma)=\bigl(\Nrd_A(\alpha)\bigr)$, so $e_1(\sigma)=0$.

On the other hand, taking for $A$ an indecomposable algebra of degree~$8$ yields examples where (b) holds but (a) does not (see~\cite[Ex.~3.6]{QT}), whereas Proposition~\ref{dec.prop} shows that (a) and (b) are equivalent when $\deg A=12$. The implication (c)~$\Rightarrow$~(b) does not hold, even when $A$ is split of degree~$12$: for instance,
any quadratic form which is the orthogonal sum of a $3$-fold and a
$2$-fold Pfister form is a $12$-dimensional quadratic form with
trivial discriminant, which need not be hyperbolic over a quadratic
field extension of the base field. For an explicit example, consider
for instance $\varphi=\pi_3\oplus \pform{x,y}$ over $F=k(\!(x)\!)(\!(y)\!)$,
where $\pi_3$ is an arbitrary anisotropic $3$-fold Pfister form
over $k$.

Note also that Tao's computation in \cite{Tao} shows that when (a) holds, then $e_2(\sigma)$ is represented by $[H]+(d,d_0)$ where $e_1(\rho)=(d)$ and $e_1(\sigma_0)=(d_0)$. It is therefore easy to see that (a) does not imply $e_2(\sigma)=0$.
\end{rem}

By contrast, the condition $e_1(\sigma)=e_2(\sigma)=0$ turns
out to be sufficient for the existence of a quadratic extension $K$
such that $(A,\sigma)_K$ is hyperbolic (hence also for a decomposition as in Proposition~\ref{dec.prop}(i)) when $\deg A=12$. The following
result may be regarded as a generalization of Pfister's theorem on
$12$-dimensional quadratic forms with trivial discriminant and trivial
Clifford invariant.

\begin{thm}
\label{Pfister.thm}
Let $(A,\sigma)$ be a central simple algebra with orthogonal
involution of degree~$12$. The following conditions are equivalent:
\begin{enumerate}
\item[(a)]
  $e_1(\sigma)=e_2(\sigma)=0$;
\item[(b)]
  there exists a central simple algebra with orthogonal involution
  $(A_0,\sigma_0)$ of degree~$6$ and a quaternion algebra with
  orthogonal involution $(H,\rho)$ such that, writing $e_1(\rho)=(d)$
  and $e_1(\sigma_0)=(d_0)$,
  \[
    (A,\sigma)=(A_0,\sigma_0)\otimes(H,\rho) \quad\text{and}\quad
    H=(d,d_0).
  \]
\end{enumerate} 
\end{thm}

\begin{proof}
  That (b) implies (a) follows from the computation of the
  discriminant and the Clifford algebra of decomposable algebras with
  involution, see \cite[(7.3)]{KMRT} and~\cite{Tao}.

  %For the converse,
  %assume~(a) holds. If $A$ is split, then $\sigma$ is adjoint to a
  %quadratic form of dimension~$12$ in $I^3F$. The anisotropic kernel
  %of such a form has dimension~$0$, $8$ or $12$, hence its Witt index
  %is not odd. Whether $A$ is split or not, as explained in~\cite{GaQ},
  %Garibaldi's open orbit argument shows that there exists a quadratic
  The first part of the argument for the converse is borrowed from ~\cite{GaQ}. More precisely, assume condition (a) holds. Then one of the half-spin representations $V$ of $\Spin(A,\sigma)$ is defined over $F$. By a classical result in representation theory, since the degree of $A$ is $12$, $\Spin(A,\sigma)$ has an open orbit in ${\mathbb P}(V)(F_{\mathrm {alg}})$, where $F_{\mathrm{alg}}$ is an algebraic closure of $F$. Using this open orbit, Garibaldi produced in (loc. cit., proof of Th.~3.1) a quadratic field extension
  $K=F(\sqrt d)$ of $F$ over which $\sigma$ is
  hyperbolic. Therefore, Proposition~\ref{dec.prop} applies and yields
  a decomposition
  \[
    (A,\sigma)=(A_0,\sigma_0)\otimes(H,\rho)
  \]
  for some algebra with orthogonal involution $(A_0,\sigma_0)$ of
  degree~$6$ and some quaternion algebra with orthogonal involution
  $(H,\rho)$ such that $e_1(\rho)=(d)$. Let
  $e_1(\sigma_0)=(d_0)$. Tao's 
  computation in~\cite{Tao} shows that the Clifford algebra of
  $(A,\sigma)$ has two components, which are Brauer-equivalent to
  $[H]+(d,d_0)$ and $[A_0]+(d,d_0)$. Therefore, the triviality of
  $e_2(\sigma)$ implies that $(d,d_0)=[H]$ or $[A_0]$. The proof is
  complete if the first equation holds.

  For the rest of the proof, assume $(d,d_0)=[A_0]$. Then $K$ splits
  $A_0$ as well as $H$, hence it splits~$A$. Therefore, by
  Proposition~\ref{dec.prop}, we may assume
  $(A_0,\sigma_0)=\Ad_\varphi$ for some $6$-dimensional quadratic form
  $\varphi$. Let $\qform{\lambda_1,\ldots,\lambda_6}$ be a
  diagonalization of $\varphi$ and let $q\in H$ be such that
  $\rho(x)=q\overline{x}q^{-1}$ for $x\in H$. Then
  $(d_0)=(-\lambda_1\cdots\lambda_6)$, $F(q)\simeq K$, and
  $(A,\sigma)\simeq\Ad_h$ for the skew-hermitian form
  $h=\qform{\lambda_1q,\ldots,\lambda_6q}$. Let $u\in H^\times$ be a
  quaternion that anticommutes with $q$, and let $c=u^2\in
  F^\times$. Then $[H]=(c,d)$ and
  \[
    \overline{ux}\cdot q\cdot ux = \overline{x}\cdot cq\cdot x
    \qquad\text{for $x\in H$},
  \]
  hence the skew-hermitian forms $\qform{q}$ and $\qform{cq}$ are
  isometric. Therefore, 
  \[
    h\simeq\qform{\lambda_1q,\ldots,\lambda_5q,c\lambda_6q} \simeq
    \varphi'\otimes\qform{q}
    \quad\text{for
  $\varphi'=\qform{\lambda_1,\ldots,\lambda_5,c\lambda_6}$,}
  \]
  and we have another decomposition
  \[
    (A,\sigma)\simeq{\Ad_{\varphi'}}\otimes(H,\rho),
    \qquad\text{with $e_1(\varphi')=(cd_0)$.}
  \]
  Since $(d,d_0)=[A_0]=0$ and $[H]=(c,d)$, it follows that
  $\bigl(e_1(\varphi'), e_1(\rho)\bigr)=[H]$, hence the latter
  decomposition satisfies the conditions in~(b). 
\end{proof}

To emphasize the analogy between Theorem~\ref{Pfister.thm} and
Pfister's result in \cite[pp.~123--124]{Pfister}, we derive an
additive decomposition of 
$(A,\sigma)$ from the multiplicative decomposition in
Theorem~\ref{Pfister.thm}(b). Since $\deg A_0=6$ and $2[A_0]=0$, there
is a quaternion algebra $H'$ such that $A_0\simeq M_3(H')$. The
involution $\sigma_0$ is adjoint to some skew-hermitian form $h$ over
$(H',\ba)$. Pick a diagonalization $h=\qform{q_1,q_2,q_3}$, for some
pure quaternions $q_i\in H'$. Denote $a_i=q_i^2$, and consider $b_i\in
F^\times$ for $i=1$, $2$, $3$ such that
$H'=(a_1,b_1)=(a_2,b_2)=(a_3,b_3)$. Since $e_1(\sigma_0)=d_0$, we have
$(a_1a_2a_3)=(d_0)$. The algebra with involution $(M_3(H'),\ad_h)$ is
an orthogonal sum of the $(H',\rho_i)$, where $\rho_i=\Int(q_i)\circ
\ba$ has discriminant $a_i$. This yields an additive decomposition of
$(A,\sigma)$, namely (in the notation of \cite[\S3.1]{QT:Peyre})
\begin{equation}
  \label{eq:addec}
  (A,\sigma)\in\boxplus_{i=1}^3 (H',\rho_i)\otimes (H,\rho).
\end{equation}
Each term in this decomposition is a central simple algebra of
degree~$4$ with orthogonal involution of trivial discriminant. It can
be rewritten as a tensor product of two quaternion algebras with
canonical involution
\begin{equation}
  \label{eq:chgbase}
  (H',\rho_i)\otimes(H,\rho) \simeq (H_i,\ba)\otimes (Q_i,\ba)
\end{equation}
with $H_i=(a_id_0,d)$ and $Q_i=(a_i,b_id)$. (This follows from a
calculation of Clifford algebras or, more elementarily, from a
suitable choice of base change.) We thus recover the decomposition in
Corollary~3.3 of \cite{QT:Peyre}.

If $A$ is split, hence $(A,\sigma)=\Ad_\psi$ for some $12$-dimensional
form $\psi$ of trivial discriminant and Clifford invariant, then
$H\simeq H'$, hence each term on the right side of~\eqref{eq:addec}
can be written as $\Ad_{\pi_i}$ for some $2$-fold Pfister form
$\pi_i$, and~\eqref{eq:addec} yields
\begin{equation}
  \label{eq:addec2}
  \psi\simeq \qform{\alpha_1}\pi_1 \perp \qform{\alpha_2}\pi_2 \perp
  \qform{\alpha_3}\pi_3
\end{equation}
for some $\alpha_1$, $\alpha_2$, $\alpha_3\in F^\times$. We thus get a
decomposition of $\psi$ as in~\cite[p.~124]{Pfister}. Note moreover
that each summand $(H',\rho_i)\otimes (H,\rho)$ becomes hyperbolic
over $K=F(\sqrt d)$, hence $\pi_i\simeq\pform{\beta_i,d}$ for some
$\beta_i\in F^\times$. Since
$e_2(\psi)=e_2(\pi_1)+e_2(\pi_2)+e_3(\pi_3)=0$, we may assume
$(\beta_1\beta_2\beta_3)=0$. Equation~\eqref{eq:addec2} can be
rewritten as 
\begin{equation}
  \label{eq:Pfister}
  \psi\simeq(\qform{\alpha_1}\pform{\beta_1} \perp
  \qform{\alpha_2}\pform{\beta_2} \perp
  \qform{\alpha_3}\pform{\beta_3})\otimes \pform{d} \qquad\text{with}\quad
    (\beta_1\beta_2\beta_3)=0.
\end{equation}

\section{Applications}

\subsection{Existence of orthogonal involutions with trivial invariants}

As a corollary of Theorem~\ref{Pfister.thm}, we characterize the
biquaternion algebras $D$ such that $M_3(D)$ carries an orthogonal
involution with trivial discriminant and Clifford invariant.

\begin{cor}
  \label{cor:existence}
  Let $D$ be a biquaternion $F$-algebra. There exists an orthogonal
  involution on $M_3(D)$ having trivial discriminant and trivial
  Clifford invariant if and only if $D$ admits a decomposition into
  quaternion algebras $D=H'\otimes H$ such that the reduced norm
  $n_{H'}$ and the pure subform $n^0_H$ of the reduced norm $n_H$
  (i.e., its restriction to the pure quaternions) have a common
  nonzero value. 

  If $I^3F=0$, this condition holds for every
  biquaternion $F$-algebra $D$.
\end{cor}

\begin{proof} 
Assume first there exists an orthogonal involution $\sigma$ on
$M_3(D)$ which has trivial discriminant and trivial Clifford
invariant. The algebra with involution $(M_3(D),\sigma)$ admits a
decomposition as in Theorem~\ref{Pfister.thm}, with $A_0=M_3(H')$ for
some quaternion algebra $H'$. Consider the discriminant $d_0$ of the
involution $\sigma_0$. We have $d_0=-\Nrd_{M_3(H')}(s)$, where $s\in
M_3(H')$ is any invertible skew-symmetric element, hence $d_0$ is a
value of $n_{H'}$ by \cite[Lemma~2.6.4]{GS}. In addition, $d_0=j^2$
for some pure quaternion $j\in H=(d,d_0)$. Therefore $d_0=-n_H^0(j)$,
so the quadratic forms $n_{H'}$ and $n_H^0$ share $-d_0$ as a
common nonzero value.   

To prove the converse, assume $D=H'\otimes H$ for some quaternion
algebras $H$ and $H'$, such that there exists a quaternion
$q\in H'$ and a pure quaternion $j\in H$ satisfying
$n_{H'}(q)=n_H(j)\neq0$. Let $d_0= j^2=-n_{H'}(q)$, and let
$H'_0\subset H'$ be the vector subspace of pure quaternions. Pick an
arbitrary invertible $q_3\in H'_0$. The vector space
$qq_3^{-1}H'_0\subset H'$ has dimension~$3$, hence
\[
  \dim(qq_3^{-1}H'_0\cap H'_0)\geq2.
\]
Since $\dim H'_0=3$, the Witt index of $n_{H'}^0$ is at most~$1$,
hence $qq_3^{-1}H'_0\cap H'_0$ contains anisotropic
vectors. Therefore, there exist $q_1$, $q_2\in H'_0$ invertible such
that $qq_3^{-1}q_2^{-1}=q_1$, i.e., $q=q_1q_2q_3$. Then
$d_0=-n_{H'}(q)$ is the discriminant of the involution adjoint to
the skew-hermitian form $h=\qform{q_1,q_2,q_3}$ over $(H',\ba)$. Pick
a pure quaternion $i$ which anticommutes with $j$, and define
$\rho=\Int(i)\circ \ba$. We get that $H=(d,d_0)$, where
$d=i^2=-n_H(i)$ is the discriminant of the orthogonal involution
$\rho$ on $H$. The involution $\sigma={\ad_h}\otimes\rho$ on
$M_3(H')\otimes H=M_3(D)$ satisfies
\[
  (M_3(D),\sigma) = (M_3(H'),{\ad_h})\otimes(H,\rho)
  \qquad\text{with}\quad [H]=(d,d_0).
\]
Therefore, Theorem~\ref{Pfister.thm} shows that
$e_1(\sigma)=e_2(\sigma)=0$.

If $I^3(F)=0$, then the reduced norm form of every quaternion algebra
represents every nonzero element in $F$, hence the condition holds for
every biquaternion $F$-algebra $D$.
\end{proof}

\begin{ex}
  \label{ex:existence}
  Let $F_0$ be an arbitrary field of characteristic different
  from~$2$, and let $F=F_0(\!(x_1)\!)(\!(y_1)\!)(\!(x_2)\!)(\!(y_2)\!)$ be the field
  of iterated Laurent series in four variables over $F_0$. The
  biquaternion algebra $D=(x_1,y_1)\otimes(x_2,y_2)$ carries a unique
  valuation $v$ extending the $(x_1,\ldots,y_2)$-adic valuation on
  $F$, and it is totally ramified over $F$. We claim that $M_3(D)$
  does not carry any orthogonal involution with trivial discriminant
  and trivial Clifford invariant. To see this as a consequence of
  Corollary~\ref{cor:existence}, consider a decomposition $D=H'\otimes
  H$ into quaternion subalgebras. Let $\Gamma_D$, $\Gamma_{H'}$,
  $\Gamma_H$, $\Gamma_F$ be the value groups of $D$, $H'$, $H$, $F$
  for the valuation $v$, so $\Gamma_F=\mz^4$ and
  $\Gamma_D=(\frac12\mz)^4$. By \cite[Cor.~8.11]{TW} we have
  $\Gamma_D/\Gamma_F=(\Gamma_{H'}/\Gamma_F)\oplus(\Gamma_H/\Gamma_F)$,
  hence $\Gamma_{H'}\cap\Gamma_H=\Gamma_F$. For $x\in {H'}^\times$ we
  have $v(x)=\frac12v\bigl(n_{H'}(x)\bigr)$ by~\cite[Th.~1.4]{TW},
  hence $v\bigl(n_{H'}(x)\bigr)\in2\Gamma_{H'}$. Similarly,
  $v\bigl(n_H(y)\bigr)\in2\Gamma_H$ for $y\in H^\times$. But the
  valuation on $H$ is an ``armature gauge'' as defined
  on~\cite[p.~339]{TW}, which means that for every standard quaternion
  basis $1$, $i$, $j$, $k$ of $H$ and $\lambda_0$, \ldots,
  $\lambda_3\in F$ 
  \[
    v(\lambda_0+\lambda_1i+\lambda_2j+\lambda_3k) =
    \min\{v(\lambda_0), v(\lambda_1i), v(\lambda_2j), v(\lambda_3k)\}.
  \]
  Since $H$ is totally ramified over $F$, $v(1)$, $v(i)$, $v(j)$, and
  $v(k)$ are in different cosets of $\Gamma_D$ modulo $\Gamma_F$.
  Therefore, if $y\in H^\times$ is a pure quaternion, then
  $v(y)\notin\Gamma_F$, hence $v\bigl(n_H(y)\bigr)\in2\Gamma_H
  \setminus 2\Gamma_F$. In conclusion, it is impossible to find
  $x\in{H'}^\times$ and $y\in H_0$ such that $n_{H'}(x)=n_H(y)$,
  because $2\Gamma_{H'}\cap2\Gamma_H = 2\Gamma_F$.
\end{ex}

\subsection{A formula for the $f_3$-invariant}

In the situation of Theorem~\ref{Pfister.thm}, the algebras $H$ and
$A_0$ occurring in the decomposition of $(A,\sigma)$ with
$e_1(\sigma)=e_2(\sigma)=0$ are not uniquely determined, even up to
Brauer-equivalence. Take for instance an arbitrary quaternion algebra
$H=(d,d_0)$ with an orthogonal involution $\rho$ of
discriminant~$(d)$. As $-d_0$ is represented by the reduced norm form
$n_H$, we may argue as in the proof of
Corollary~\ref{cor:existence} to find pure quaternions $q_1$, $q_2$,
$q_3\in H$ such that $n_H(q_1q_2q_3)=-d_0$. On $A_0=M_3(H)$, the
orthogonal involution $\sigma_0$ adjoint to the skew-hermitian form
$\qform{q_1,q_2,q_3}$ has discriminant $(d_0)$. Therefore,
$(A,\sigma)=(A_0,\sigma_0)\otimes(H,\rho)$ satisfies the conditions of
Theorem~\ref{Pfister.thm}. But $A$ is split since $A_0$ and $H_0$ are
Brauer-equivalent, hence $(A,\sigma)\simeq\Ad_\psi$ for some
$12$-dimensional quadratic form $\psi$ with
$e_1(\psi)=e_2(\psi)=0$. By Pfister's result (see~\eqref{eq:Pfister}),
there is a decomposition $\psi\simeq\psi_0\otimes\beta$ for some
$6$-dimensional form $\psi_0$ with $e_1(\psi_0)=0$ and some
$2$-dimensional form $\beta$, hence another decomposition of
$(A,\sigma)$ as in Theorem~\ref{Pfister.thm}:
\[
  (A_0,\sigma_0)\otimes(H,\rho) = (A,\sigma) \simeq {\Ad_{\psi_0}}
  \otimes {\Ad_\beta}.
\]

Nevertheless, we show in this section that the invariant $f_3(\sigma)$
defined in \cite[Def.~2.4]{QT:Peyre} can be calculated from any
decomposition as in Theorem~\ref{Pfister.thm}, and can thus yield some
information on the possible decompositions. The main ingredient of
the proof is Theorem 5.4 in~\cite{QT:Peyre}, which shows that
$f_3(\sigma)$ is the Arason invariant of the sum of the norm forms of
all quaternion algebras in a given decomposition group of
$(A,\sigma)$. Since the $f_3$ invariant is defined only when the
underlying central simple algebra carries a hyperbolic involution, we
need to assume in the following statement, which is the main result of
this section, that the index of $A$ is at most~$2$.  

\begin{thm} 
\label{f3.thm}
Let $(A,\sigma)$ be a central simple algebra of degree $12$ and index
$\leq 2$ with an orthogonal involution with trivial discriminant and
trivial Clifford invariant. Pick a decomposition of $(A,\sigma)$ as in
Theorem~\ref{Pfister.thm},
\[
  (A,\sigma)\simeq (A_0,\sigma_0)\otimes (H,\rho)
\]
where $(A_0,\sigma_0)$ is a central simple algebra with orthogonal
involution of degree~$6$ and $(H,\rho)$ is a quaternion algebra with
orthogonal involution, and $H=(d,d_0)$ with $e_1(\rho)=d$ and
$e_1(\sigma_0)=d_0$. Let $Q$ and $H'$ be the quaternion algebras that
are Brauer-equivalent to $A$ and $A_0$ respectively, and let $n_Q$,
$n_{H'}$, $n_H$ be the reduced norm forms of $Q$, $H'$ and $H$
respectively. With this notation,
\begin{equation}
  \label{eq:f31}
  f_3(\sigma)=e_3(n_Q-n_{H}-\qform d n_{H'})\in H^3(F).
\end{equation}
(Note that $n_Q-n_{H}-\qform d n_{H'}\in I^3F$ because
$[Q]+[H]+[H']=0$.) 
Moreover, if $c\in F^\times$ is such that $H$, $H'$ and $Q$ are all
split by $F(\sqrt{c})$, and $e\in F^\times$ is such that $H=(c,e)$,
then
\begin{equation}
  \label{eq:f32}
  f_3(\sigma)=(de)\cdot [Q] = (de)\cdot[H'].
\end{equation}
\end{thm}
 
\begin{proof}
  Consider the additive decomposition of $(A,\sigma)$
  in~\eqref{eq:addec}. Together with~\eqref{eq:chgbase}, it shows that
\[
  \{0,\,[Q],\,[Q_1],\,[H_1],\,[Q_2],\,[H_2],\,[Q_3],\,[H_3]\}
\]
  is a decomposition group of $(A,\sigma)$ as defined
  in~\cite[Def.~3.6]{QT:Peyre}. As a result, Theorem 5.4
  in~\cite{QT:Peyre} yields  
  \[
    f_3(\sigma)= e_3\bigl(n_Q+\sum_{i=1}^3 n_{H_i}+\sum_{i=1}^3
    n_{Q_i}\bigr).
  \]
  In order to compute the Arason invariant of this quadratic form, we
  use the following identity in the Witt group of $F$:  
  \[
    \pform{\lambda,\mu\nu}=
    \pform{\lambda,\mu}+\qform{\mu}\pform{\lambda,\nu}. 
  \]
 In particular, it shows that for $i=1$, $2$, and $3$, we have 
 \[
   n_{H_i}=\pform{a_i,d}+\qform{a_i}n_{H}\quad \text{and} \quad 
   n_{Q_i}=\pform{a_i,d}+\qform{d}n_{H'}.
 \]
  Therefore,
  \begin{equation}
    \label{eq:f1}
    \sum_{i=1}^3 n_{H_i}+\sum_{i=1}^3 n_{Q_i}= 
    \qform{a_1,a_2,a_3} n_H+\qform{d,d,d}n_{H'}+
    \sum_{i=1}^3\pform{-1,a_i,d}.
  \end{equation}
  Recall that $(d_0)=(a_1a_2a_3)$, hence
  \[
    \qform{a_1,a_2,a_3}n_H\equiv \qform{-d_0}n_H \bmod I^4F.
  \]
  Similarly,
  \[
    \qform{d,d,d}n_{H'} \equiv \qform{-d}n_{H'} \bmod I^4F.
  \]
  Therefore, \eqref{eq:f1} yields
  \begin{align*}
    e_3\bigl(n_Q+\sum_{i=1}^3 n_{H_i}+\sum_{i=1}^3
    n_{Q_i}\bigr) & = e_3(n_Q-\qform{d_0}n_H - \qform{d}n_{H'})
                    +\sum_{i=1}^3(-1,a_i,d)\\
    & = e_3(n_Q-n_H-\qform d n_{H'}) + (d_0)\cdot[H] + (-1,d_0,d).
  \end{align*}
  Now, since $H=(d,d_0)$ and $(d_0,d_0)=(-1,d_0)$, the last two terms
  on the right side of the last displayed equation cancel, and 
  Formula \eqref{eq:f31} is proved.

  To obtain Formula~\eqref{eq:f32}, choose $c\in F^\times$ such that
  $F(\sqrt c)$ splits $Q$, $H$, and $H'$, and let $e$, $e'\in
  F^\times$ be such that $H=(c,e)$ and $H'=(c,e')$, hence
  $Q=(c,ee')$. Then
  \begin{align*}
    n_Q-n_H-\qform d n_{H'} & = \pform{c,ee'} - \pform{c,e} - \qform d
                              \pform{c,e'}\\
    & = \pform c \qform{e,-ee',-d,de'}\\
    & = \qform e \pform{c,e',de}.
  \end{align*}
  Therefore, $f_3(\sigma)=(c,e',de)=(de)\cdot [H']$. As
  $H=(c,e)=(d,d_0)$, we have
  \[
  (d)\cdot [H] = (-1)\cdot [H] = (e)\cdot [H],
  \]
  hence $(de)\cdot [H]=0$ and $(de)\cdot [H'] = (de)\cdot
  [Q]$. Formula~\eqref{eq:f32} is thus proved. 
\end{proof}
 
\begin{cor}
  With the notation of Theorem~\ref{f3.thm}, we have $f_3(\sigma)=0$
  if any of the following conditions holds:
  \begin{enumerate}
  \item[(i)]
    $A$ is split;
  \item[(ii)]
    $A_0$ is split;
  \item[(iii)]
    $A_0$ is split by $F(\sqrt{d_0})$.
  \end{enumerate} 
\end{cor} 

\begin{proof} 
  Formula~\eqref{eq:f32} readily shows that $f_3(\sigma)=0$ when~(i)
  or (ii) holds. In case~(iii) we may take $c=d_0$ and $e=d$ in
  Formula~\eqref{eq:f32} to obtain $f_3(\sigma)=0$.

  Alternatively, in case~(i) we may argue that $(A,\sigma)=\Ad_\psi$
  for some quadratic form $\psi\in I^3F$, hence
  $e_3(\sigma)=e_3(\psi)\in H^3(F,\mu_2)$ and therefore
  $f_3(\sigma)=0$ by definition. Also, in case~(ii) $(A,\sigma)$ is
  split and hyperbolic over $F(\sqrt d)$, hence $f_3(\sigma)=0$ by
  \cite[Prop.~5.6]{QT:Peyre}.
\end{proof} 
 
By contrast, $f_3(\sigma)$ does not necessarily vanish when $H$ is
split. In that case we may choose $e=1$ in Formula~\eqref{eq:f32} and
derive the following: if $(A,\sigma)=(A_0,\sigma_0)\otimes\Ad_{\pform
  d}$ and $e_1(\sigma_0)=(d_0)$ is such that $(d,d_0)$ is split, then
\[
  f_3(\sigma)=(d)\cdot [A_0].
\]
This also follows from~\cite[Cor.~2.18]{QT:Peyre}.

\end{document}